\documentclass[draft,reqno]{amsproc}
\usepackage{amssymb}
\usepackage{euscript} %txfonts,pxfonts
\usepackage{mathrsfs} %\mathscr T
\usepackage{color}
\makeatletter
\@namedef{subjclassname@2010}{%
\textup{2010} Mathematics Subject Classification}
\makeatother

\numberwithin{equation}{section}

\newtheorem{thm}{Theorem}[section]
\newtheorem{cor}[thm]{Corollary}
\newtheorem{lem}[thm]{Lemma}
\newtheorem{pro}[thm]{Proposition}

\newtheorem*{thm*}{Theorem}

\newtheorem*{opq*}{Problem}

\theoremstyle{remark}
\newtheorem{rem}[thm]{Remark}

\theoremstyle{definition}
\newtheorem{exa}[thm]{Example}

\DeclareMathOperator{\E}{e}

\newcommand*{\bscr}{\mathscr{B}}

\newcommand*{\cbb}{\mathbb C}
\newcommand*{\cscr}{\mathscr{C}}

\newcommand*{\dbb}{{\mathbb D}}

\newcommand*{\dz}[1]{{\mathscr D}(#1)}

\newcommand*{\gev}[2]{{\mathscr E}_{#1;#2}}
\newcommand*{\fscr}{\mathscr F}
\newcommand*{\fcal}{\mathcal F}
\newcommand*{\Ge}{\geqslant}
\newcommand*{\gammab}{\boldsymbol \gamma}
\newcommand*{\hh}{\mathcal H}

\newcommand*{\I}{{\mathrm i\hspace{.1ex}}}

\newcommand*{\is}[2]{\langle#1,#2\rangle}

\newcommand*{\jd}[1]{\mathscr N(#1)}
\newcommand*{\kbb}{\mathbb K}

\newcommand*{\Le}{\leqslant}
\newcommand*{\lino}[1]{\boldsymbol L(#1)}
\newcommand*{\mscr}{\mathscr{M}}

\newcommand*{\nbb}{\mathbb N}
\newcommand*{\nil}[1]{\mathsf{n}(#1)}

\newcommand*{\ogr}[1]{\boldsymbol B(#1)}

\newcommand*{\rbb}{\mathbb R}

\newcommand*{\tbb}{{\mathbb T}}
\newcommand*{\tscr}{{\mathscr T}}
\newcommand*{\ubb}{\mathbb{U}}

\newcommand*{\xscr}{{\mathscr X}}
\newcommand*{\zbb}{\mathbb Z}

\begin{document}
   \title[$m$-isometric operators and their local properties]{$m$-isometric operators and their local properties}
   \author[Z.\ J.\ Jab{\l}o\'nski]{Zenon Jan
Jab{\l}o\'nski}
   \address{Instytut Matematyki,
Uniwersytet Jagiello\'nski, ul.\ \L ojasiewicza 6,
PL-30348 Kra\-k\'ow, Poland}
\email{Zenon.Jablonski@im.uj.edu.pl}
   \author[I.\ B.\ Jung]{Il Bong Jung}
   \address{Department of Mathematics, Kyungpook National University,
Da\-egu 41566, Korea} \email{ibjung@knu.ac.kr}
   \author[J.\ Stochel]{Jan Stochel}
\address{Instytut Matematyki, Uniwersytet
Jagiello\'nski, ul.\ \L ojasiewicza 6, PL-30348
Kra\-k\'ow, Poland} \email{Jan.Stochel@im.uj.edu.pl}
   \thanks{The research of the second author was supported by the National Research
Foundation of Korea (NRF) grant funded by the Korea
government (MSIT) (2018R1A2B6003660)}

   \maketitle
   \begin{abstract}
In this paper we give necessary and sufficient conditions
for a bounded linear Hilbert space operator to be an
$m$-isometry for an unspecified $m$ written in terms of
conditions that are applied to ``one vector at a time''. We
provide criteria for orthogonality of generalized
eigenvectors of an (a priori unbounded) linear operator $T$
on an inner product space that correspond to distinct
eigenvalues of modulus $1$. We also discuss a similar
question of when Jordan blocks of $T$ corresponding to
distinct eigenvalues of modulus $1$ are ``orthogonal''.
   \end{abstract}
   \section{\label{Sec1}Introduction and notation}
   Let $\hh$ be a Hilbert space (all vector spaces are
assumed to be complex throughout this paper). Denote by
$\ogr{\hh}$ the $C^*$-algebra of all bounded linear
operators on $\hh$ and by $I_{\hh}$ (abbreviated to $I$ if
no confusion arises) the identity operator on $\hh$.
Suppose that $T\in \ogr{\hh}$. Define
   \begin{align*}
\bscr_m(T) = \sum_{k=0}^m (-1)^k \binom{m}{k}
{T^*}^kT^k, \quad m=0,1,2,\ldots.
   \end{align*}
Given a positive integer $m$, we say that $T$ is an
{\em $m$-isometry} if $\bscr_m(T)=0$. Clearly,
   \begin{align} \label{miso-1}
   \begin{minipage}{72ex}
{\em $T$ is an $m$-isometry if and only if
$\sum\limits_{k=0}^m (-1)^k \binom{m}{k} \|T^k h\|^2 =
0$ for all $h \in \hh$.}
   \end{minipage}
   \end{align}
A $1$-isometry is an isometry and {\em vice versa}. The
notion of an $m$-isometric operator has been invented by
Agler (see \cite[p.\ 11]{Ag-0}). We refer the reader to the
trilogy \cite{Ag-St1,Ag-St2,Ag-St3} by Agler and Stankus
for the fundamentals of the theory of $m$-isometries. The
class of $m$-isometric operators emerged from the study of
the time shift operator of the modified Brownian motion
process \cite{Ag-St2} as well as from the investigation of
invariant subspaces of the Dirichlet shift
\cite{Rich88,Rich91}. The topics related to $m$-isometries
are currently being studied intensively (see e.g.,
\cite{Ber-Mar-No13,Ber-Mar-Mul-No14,Le15,Ab-Le16,McC-Ru2016,Ko-Lee2018,A-C-J-S19}).

Let us recall a few facts about $m$-isometries. It is
easily seen that $m$-isometries are injective. The
following is well known (see \cite[Lemma~1.21]{Ag-St1}).
   \begin{align} \label{as-spec}
   \begin{minipage}{70ex}
{\em If $\hh\neq \{0\}$ and $T\in\ogr{\hh}$ is an
$m$-isometry, then either $\sigma(T) \subseteq \tbb$ or
$\sigma(T)=\overline{\dbb}${\em ;} in both cases $r(T)=1$.}
   \end{minipage}
   \end{align}
Here $\sigma(T)$ is the spectrum of $T$, $r(T)$ is the
spectral radius of $T$, $\dbb$ is the open unit disc in the
complex plane centered at $0$ and $\tbb$ is its topological
boundary. Following \cite{Bo-Ja}, we say that an
$m$-isometry $T$ is {\em strict} if $m=1$ and $\hh\neq
\{0\}$, or $m\Ge 2$ and $T$ is not an $(m-1)$-isometry (in
both cases $\hh\neq \{0\}$). Examples of strict
$m$-isometries for any $m\Ge 2$ are provided in
\cite[Proposition~8]{At91} (see also
\cite[Example~2.3]{Sh-At}). As observed by Agler and
Stankus, if $T\in \ogr{\hh}$ is an $m$-isometry, then $T$
is a $k$-isometry for all $k\Ge m$ (see line 6 on page 389
in \cite{Ag-St1} and note that $\beta_m(T)=
\frac{(-1)^m}{m!} \bscr_m(T)$). This implies that if $T$ is
an $m$-isometry, then there exists a unique $k\in \nbb$
such that $T$ is a strict $k$-isometry and $k\Le m$.

The well-known characterization of $m$-isometries due to
Agler and Stankus states that an operator $T\in \ogr{\hh}$
is an $m$-isometry if and only if for every $h\in \hh$,
$\|T^nh\|^2$ is a polynomial in $n$ of degree at most $m-1$
(see \cite[p.\ 389]{Ag-St1}; see also Theorem~\ref{Newt4}).
The first question we deal with in this paper is whether
dropping the degree constraint yields $m$-isometricity of
$T$ for some unspecified $m$. This problem is formulated in
the spirit of Kaplansky's theorem \cite[Theorem~15]{Kap}
which asserts that an operator $T\in \ogr{\hh}$ is
algebraic if and only if it is locally algebraic. We answer
the above question in the affirmative (see
Theorem~\ref{weakmiso}). However, if we drop the assumption
that $\hh$ is complete the answer is no (see
Remark~\ref{wiel-nod}). Nevertheless, the assertion
\eqref{miso-1} enables one to define $m$-isometries in the
case of lack of completeness of $\hh$. Motivated by needs
of the theory of unbounded operators (see e.g.,
\cite{Ja-St01}), we consider $m$-isometries $T$ on an inner
product space $\mscr$ starting from Section~ \ref{Sec4}.
Notice that any linear operator $T$ on $\mscr$ can be
regarded as an (a priori unbounded) operator in $\hh$ with
invariant domain $\mscr$, where $\hh$ is the Hilbert space
completion of $\mscr$. We do not assume that $\mscr$ is
invariant for $T^*$. Let us point out that even under these
circumstances it may happen that $\dz{T^*}=\{0\}$ (see
\cite[Example~3, p.\ 69]{Weid80}). In this part of the
paper we are mostly interested in finding criteria for
orthogonality of generalized eigenvectors of $T$
corresponding to distinct eigenvalues of modulus $1$. A
similar question can be asked about ``orthogonality'' of
Jordan blocks of $T$ corresponding to distinct eigenvalues
of modulus $1$. We answer both these questions by using the
polynomial approach as developed for $m$-isometries. The
most delicate step in the proof depends heavily on the
celebrated Carlson's theorem which gives a sufficient
condition for an entire function of exponential type
vanishing on nonnegative integers to vanish globally. These
two problems are related to \cite{A-H-S} where the question
of orthogonality of spectral spaces corresponding to
distinct eigenvalues were studied in the context of bounded
algebraic Hilbert space operators that are roots of
polynomials in two variables. As a byproduct, we answer the
question of when an algebraic operator on $\mscr$ is an
$m$-isometry. Namely, we show that for a given $m$,
algebraic $m$-isometries are precisely finite orthogonal
sums of operators of the form $z I + N$, where $z$ is a
complex number of modulus $1$ and $N$ is a nilpotent
operator with prescribed index of nilpotency (see
Proposition~\ref{m-iso-alg}). What is more, if an algebraic
operator is a strict $m$-isometry, then $m$ must be a
positive odd number. It is worth mentioning that there are
$m$-isometries and nilpotent operators that are unbounded
and closed as operators in $\hh$ (see
\cite[Example~6.4]{Ja-St01} and \cite[Example~3.3]{Ota88},
resp.). We also show that if $T\in\ogr{\hh}$ is a compact
$m$-isometry, then $\hh$ is finite dimensional (see
Proposition~\ref{compop}).

The organization of this paper is as follows. In
Section~\ref{Sec2} we provide necessary facts concerning
$m$-isometries needed in this paper. In particular, we
discuss a finite difference operator which plays an
essential role in our investigations. In Section~\ref{Sec3}
we state and prove a few ``local'' characterizations of
$m$-isometries including that mentioned above (see
Theorem~\ref{weakmiso}). These characterizations are stated
in terms of conditions that are applied to ``one vector at
a time''. In particular, we show that verifying
$m$-isometricity for some unspecified $m$ can be reduced to
considering the restrictions of the operator in question to
its cyclic invariant subspaces. This result resembles the
analogical ones for subnormal operators (see
\cite[Corollary~1]{Trent1981}; see also \cite{Sto-Sz1984}).
Starting from Section~ \ref{Sec4}, we investigate
$m$-isometries on an inner product space $\mscr$. Theorem~
\ref{Bang1-m} is an adaptation of a result due to
Berm\'{u}dez, Martin\'{o}n, M\"{u}ller and Noda
\cite[Theorem~3]{Ber-Mar-Mul-No14} to the context of linear
operators on $\mscr$. The main result of
Section~\ref{Sec5}, Theorem~\ref{alg-m-iso-lem} provides a
criterion for orthogonality of generalized eigenvectors.
Finally, in Section~\ref{Sec6}, we characterize
``orthogonality'' of Jordan blocks corresponding to
distinct eigenvalues of modulus $1$ (see Theorem~
\ref{alg-m-iso-lem-b}).

We conclude this section by fixing notation. In what
follows, $\rbb$ and $\cbb$ stand for the fields of real and
complex numbers, respectively. Set $\tbb=\{z\in \cbb\colon
|z|=1\}$ and $\ubb=\{-1,1\} \times \{-\I,\I\}$, where $\I$
stands for the imaginary unit. The sets of integers,
nonnegative integers and positive integers are denoted by
$\zbb$, $\zbb_+$ and $\nbb$, respectively. The ring of all
polynomials in one indeterminate $x$ with coefficients in a
ring $R$ is denoted by $R[x]$. If $p=\sum_{j=0}^n a_j x^j
\in \cbb[x]$, then the polynomials $p^*, \, \mathrm{Re}\,
p, \, \mathrm{Im}\, p \in \cbb[x]$ are defined by
   \begin{align*}
p^*=\sum_{j=0}^n \bar a_j x^j, \quad \mathrm{Re}\,
p=\sum_{j=0}^n \mathrm{Re} (a_j) x^j \quad \text{and} \quad
\mathrm{Im}\, p=\sum_{j=0}^n \mathrm{Im} (a_j) x^j.
   \end{align*}
As usual, the identity transformation on a vector space
$\mscr$ is denoted by $I_{\mscr}$ and abbreviated to $I$ if
no confusion arises. Given a linear map $T\colon \mscr \to
\mscr$ and a vector $h\in \mscr$, we write $\cscr_T(h)$ for
the vector space spanned by $\{T^n h\colon n\in \zbb_+\}$.
   \section{\label{Sec2}Basic characterizations of $m$-isometries}
This section provides necessary facts on a finite
difference operator that will be used in
characterizing $m$-isometries.

If $k\in \{-\infty\} \cup \zbb_+$, $R = \cbb$ or
$R=\ogr{\hh}$ and $\{\gamma_n\}_{n=0}^{\infty} \subseteq
R$, then we say that $\gamma_n$ is a {\em polynomial in
$n$} (of {\em degree} $k$) if there exists a polynomial
$p\in R[x]$ (of degree $k$) such that $p(n)=\gamma_n$ for
all $n\in \zbb_+$. It follows from the Fundamental Theorem
of Algebra that such $p$ is unique. Denote by
$\cbb^{\zbb_+}$ the vector space of all complex sequences
$\{\gamma_n\}_{n=0}^{\infty}$ with linear operations
defined coordinatewise. We regard $\cbb^{\zbb_+}$ as a
topological vector space equipped with the topology of
pointwise convergence. Note that $\cbb^{\zbb_+}$ is, in
fact, a Frech\'{e}t space (see \cite[Remarks 1.38(c)]{Rud};
see also \cite[Ex.\ 13, p.\ 104]{con2}). Define the linear
transformation $\tscr\colon \cbb^{\zbb_+} \to
\cbb^{\zbb_+}$ by
   \begin{align*}
(\tscr \gammab)_n = \gamma_{n+1}, \quad n\in \zbb_+,
\, \gammab=\{\gamma_n\}_{n=0}^{\infty} \in
\cbb^{\zbb_+}.
   \end{align*}
Set $\triangle = \tscr - I$. It is easily seen that the
transformations $\tscr$ and $\triangle$ are continuous.
Applying Newton's binomial formula, we get
   \begin{align} \label{Newt1}
(\triangle^m \gammab)_n = (-1)^m \sum_{k=0}^m (-1)^k
\binom{m}{k} \gamma_{n+k}, \quad m,n\in\zbb_+, \,
\gammab=\{\gamma_n\}_{n=0}^{\infty} \in \cbb^{\zbb_+}.
   \end{align}
Using Newton's binomial formula again, we can
reproduce the original sequence $\gammab$ by means of
$\{(\triangle^k \gammab)_0\}_{k=0}^{\infty}$ as
follows:
   \allowdisplaybreaks
   \begin{align} \notag
\gamma_n = (\tscr^n \gammab)_0 & = ((\triangle + I)^n
\gammab)_0
   \\ \notag
& = \sum_{k=0}^n \binom{n}{k} (\triangle^{k}
\gammab)_0
      \\ \label{Newt2}
& = \sum_{k=0}^{\infty} \frac{(\triangle^{k}
\gammab)_0}{k!} (n)_k, \quad n\in\zbb_+, \,
\gammab=\{\gamma_n\}_{n=0}^{\infty} \in \cbb^{\zbb_+},
   \end{align}
where $(n)_k$ is a polynomial in $n$ of degree $k$ given by
   \begin{align*}
(n)_k =
   \begin{cases}
1 & \text{if } k=0 \text{ and } n \in \zbb_+,
   \\
\prod_{j=0}^{k-1} (n-j) & \text{if } k\in \nbb \text{
and } n \in \zbb_+.
   \end{cases}
   \end{align*}
Observe that $(n)_k=0$ for all $k>n$. The formula
\eqref{Newt2}, which is known as Newton's interpolation
formula, enables as to describe the kernel $\ker
\triangle^m$ of $\triangle^m$ (cf.\ \cite[Exercise
7.2]{Dick}).
   \begin{pro} \label{kerdelta}
If $m\in \nbb$, then
   \begin{align*}
\ker \triangle^m = \Big\{\{\gamma_n\}_{n=0}^{\infty}
\in \cbb^{\zbb_+}\colon \text{$\gamma_n$ is a
polynomial in $n$ of degree at most $m-1$}\Big\}.
   \end{align*}
   \end{pro}
   \begin{proof}
The inclusion ``$\subseteq$'' is a direct consequence
of \eqref{Newt2}. In turn, the inclusion
``$\supseteq$'' can be proved by induction on $m$ by
using the fact that if $\gammab =
\{\gamma_n\}_{n=0}^{\infty} \in \cbb^{\zbb_+}$ and
$\gamma_n$ is a polynomial in $n$ of degree $m$, then
$(\triangle \gammab)_n$ is a polynomial in $n$ of
degree $m-1$.
   \end{proof}
Now we apply the above formulas to operators. For
this, we attach to $T\in \ogr{\hh}$ and $h\in \hh$,
the sequence $\gammab_{T,h} =
\{(\gammab_{T,h})_n\}_{n=0}^{\infty}$ defined by
   \begin{align*}
(\gammab_{T,h})_n=\|T^nh\|^2, \quad n\in \zbb_+.
   \end{align*}
Note that the formula (iii) of Proposition~\ref{Newt3}
below is due to Agler and Stankus (see \cite[Eq.\
(1.3)]{Ag-St1}).
   \begin{pro}\label{Newt3}
If $T\in \ogr{\hh}$, then the following formulas
hold{\em :}
   \begin{enumerate}
   \item[(i)] $(\triangle^m \gammab_{T,h})_n  = (\triangle^m
\gammab_{T,T^nh})_0$ for all $h \in \hh$ and $m,n\in
\zbb_+$,
   \item[(ii)]
$(\triangle^m \gammab_{T,h})_0 = (-1)^m\is{\bscr_m(T)h}{h}$
for all $h \in \hh$ and $m\in \zbb_+$,
   \item[(iii)] $T^{*n}T^n  = \sum_{k=0}^{\infty} (n)_k
\frac{(-1)^k}{k!} \bscr_k(T)$ for all $n\in \zbb_+$.
   \end{enumerate}
   \end{pro}
   \begin{proof}
Applying \eqref{Newt1} to $\gammab_{T,h}$, we obtain
(i) and (ii). In turn, applying \eqref{Newt2} to
$\gammab_{T,h}$ and using (ii), we get (iii).
   \end{proof}
Combining the conditions (i) and (ii) of
Proposition~\ref{Newt3}, we get the following
important property of $m$-isometries.
   \begin{cor}[\mbox{\cite[p.\ 389]{Ag-St1}}] \label{Agl-St}
If $m\in \nbb$ and $T\in \ogr{\hh}$ is an
$m$-isometry, then $T$ is a $k$-isometry for every
integer $k \Ge m$.
   \end{cor}
The rest of this section is devoted to characterizing
the class of $m$-isometric operators. Note that the
equivalence (i)$\Leftrightarrow$(ii) of
Theorem~\ref{Newt4} below is due to Agler and Stankus
(see \cite[p.\ 389]{Ag-St1}). Denote by
$\{e_n\}_{n=0}^{\infty}$ the standard orthonormal
basis of $\ell^2$, the Hilbert space of all square
summable complex sequences indexed by $\zbb_+$. For a
given bounded sequence $\{\lambda_n\}_{n=0}^{\infty}
\subseteq \cbb$ there exists a unique operator
$W\in\ogr{\ell^2}$, called a {\em unilateral weighted
shift} with weights $\{\lambda_n\}_{n=0}^{\infty}$,
such that
   \begin{align*}
We_n = \lambda_n e_{n+1}, \quad n\in \zbb_+.
   \end{align*}
   \begin{thm}\label{Newt4}
If $m\in \nbb$ and $T\in \ogr{\hh}$, then the
following conditions are equivalent{\em :}
   \begin{enumerate}
   \item[(i)] $T$ is an $m$-isometry,
   \item[(ii)] $T^{*n}T^n$ is
a polynomial in $n$ of degree at most $m-1$,
   \item[(iii)] for each $h\in \hh$,
$\|T^nh\|^2$ is a polynomial in $n$ of degree at most
$m-1$,
   \item[(iv)] $T$ is injective and for each nonzero $h\in
\hh$, the unilateral weighted shift $W_{T,h}$ with
weights\footnote{\;$W_{T,h}$ is bounded because the
sequence of its weights is bounded above by $\|T\|$.}
$\big\{\frac{\|T^{n+1}h\|}{\|T^n h\|}\big\}_{n=0}^{\infty}$
is an $m$-isometry.
   \end{enumerate}
Moreover, if $T$ is an $m$-isometry, then
   \begin{align*}
T^{*n}T^n = \sum_{k=0}^{m-1} (n)_k \frac{(-1)^k}{k!}
\bscr_k(T), \quad n\in \zbb_+.
   \end{align*}
   \end{thm}
   \begin{proof}
The implication (i)$\Rightarrow$(ii) and the
``moreover'' part are direct consequences of
Proposition~\ref{Newt3}(iii) and
Corollary~\ref{Agl-St}.

(ii)$\Rightarrow$(iii) Obvious.

(iii)$\Rightarrow$(i) Use Propositions~\ref{kerdelta}
and \ref{Newt3}(ii).

(i)$\Leftrightarrow$(iv) Apply
\cite[Proposition~2.4]{Ja-Ju-St}.
   \end{proof}
As shown below, the equivalence (i)$\Leftrightarrow$(iii)
of Corollary~\ref{TrieuLe}, which is due to Abdullah and Le
\cite[Theorem~2.1]{Ab-Le16}, is an almost immediate
consequence of the Agler-Stankus assertion ``Hence if
$T\in\mathcal L(\hh)$ one sees via (1.3) that $T$ is an
$m$-isometry if and only if $s_T(k)$ is a polynomial in $k$
of degree less than $m$'' (see \cite[p.\ 389]{Ag-St1}).
   \begin{cor} \label{TrieuLe}
Fix $m\in \nbb$. Let $W\in \ogr{\ell^2}$ be a unilateral
weighted shift with weights
$\{\lambda_n\}_{n=0}^{\infty}\subseteq \cbb$. Then the
following conditions are equivalent{\em :}
   \begin{enumerate}
   \item[(i)] $W$ is an $m$-isometry,
   \item[(ii)] $\|W^n e_0\|^2$ is a
polynomial in $n$ of degree at most $m-1$, where
$e_0=(1,0,0,\ldots)$,
   \item[(iii)] there exists a polynomial
$p\in \cbb[x]$ of degree at most $m-1$ such that
   \begin{align} \label{Frank}
\text{$p(n) > 0$ and $|\lambda_n|^2 =
\frac{p(n+1)}{p(n)}$ for all $n\in \zbb_+$.}
   \end{align}
   \end{enumerate}
   \end{cor}
   \begin{proof}
(i)$\Rightarrow$(ii) Apply Theorem~\ref{Newt4}(iii).

(ii)$\Rightarrow$(iii) Let $p\in \cbb[x]$ be a polynomial
of degree at most $m-1$ such that $\|W^n e_0\|^2 = p(n)$
for all $n\in \zbb_+$. Since $m$-isometries are injective,
$p(n) > 0$ for all $n \in \zbb_+$. This and
   \begin{align*}
|\lambda_n|^2 = \frac{\|W^{n+1} e_0\|^2}{\|W^n
e_0\|^2}, \quad n\in \zbb_+,
   \end{align*}
implies (iii).

(iii)$\Rightarrow$(i) Denote by
$\{e_n\}_{n=0}^{\infty}$ the standard orthonormal
basis for $\ell^2$. Since
   \begin{align*}
\|W^n h\|^2 = \sum_{j=0}^{\infty} |\is{h}{e_j}|^2
\|W^n e_j\|^2 \overset{\eqref{Frank}}=
\sum_{j=0}^{\infty} |\is{h}{e_j}|^2
\frac{p(n+j)}{p(j)}, \quad n \in \zbb_+, \, h \in
\ell^2,
   \end{align*}
and the transformation $\triangle\colon \cbb^{\zbb_+} \to
\cbb^{\zbb_+}$ is linear and continuous, we infer from
Proposition~\ref{kerdelta} that $\triangle^m \gammab_{W,h}
= 0$ for all $h\in \ell^2$. This together with
Proposition~\ref{Newt3}(ii) completes the proof.
   \end{proof}
   \section{\label{Sec3}``Local'' characterizations of $m$-isometries}
In this section we prove, among other things, that if
an operator $T\in\ogr{\hh}$ has the property that
$\|T^n h\|^2$ is a polynomial in $n$ for every $h\in
\hh$, then there exists $m\in \nbb$ such that $T$ is
an $m$-isometry (see Theorem~\ref{weakmiso} below).
Before we do this, we need three lemmata. The first is
patterned on the polarization formula for polynomials
on vector spaces (see \cite[Theorem~A]{Bo-Si}; see
also \cite{M-O}). Its routine verification is left to
the reader.
   \begin{lem} \label{polarf}
Let $\kbb\in \{\rbb,\cbb\}$, $\xscr$ be a vector space over
$\kbb$ and $\varphi\colon \xscr \times \xscr \to \kbb$ be a
mapping such that $\varphi(\cdot,h)$ and $\varphi(h,\cdot)$
are additive for every $h\in \xscr$. Then
   \begin{align*}
\varphi(h,h) = \frac{1}{2} \sum_{j=0}^2 (-1)^{j}
\binom{2}{j} \varphi(h_0+ j h, h_0+ j h), \quad h,h_0 \in
\xscr.
   \end{align*}
   \end{lem}
The second lemma gives a sufficient condition for a
sequence of bounded operators to have at least one
vanishing term.
   \begin{lem} \label{lnzero}
Let $\mathscr V$ be a nonempty open subset of $\hh$
and let $\{L_m\}_{m=1}^{\infty} \subseteq \ogr{\hh}$
be a sequence with the property that for every $h\in
\mathscr V$ there exists $m_h\in\nbb$ such that
   \begin{align*}
\is{L_{m_h}h}{h}=0.
   \end{align*}
Then there exists $m\in \nbb$ such that $L_{m}=0$.
   \end{lem}
   \begin{proof}
Define for every $k\in \nbb$ the subset $\mathscr V_k$ of
$\mathscr V$ by
   \begin{align*}
\mathscr V_k = \big\{h\in \mathscr V\colon
\is{L_{k}h}{h}=0\big\}.
   \end{align*}
Clearly, each set $\mathscr V_k$ is relatively closed
in $\mathscr V$. By assumption $\mathscr V =
\bigcup_{k=1}^{\infty} \mathscr V_k$. According to the
Baire category theorem \cite[Theorem~48.2,
Lemma~48.4]{Munkres00}, there exist $m\in \nbb$,
$h_0\in \mathscr V_{m}$ and $\epsilon \in (0,\infty)$
such that
   \begin{align} \label{miso-a}
\text{$\is{L_{m}h}{h}=0$ for every $h\in \hh$ such
that $\|h-h_0\| < \epsilon$.}
   \end{align}
Applying Lemma~\ref{polarf} to
$\varphi(f,g)=\is{L_{m}f}{g}$, we see that the
following equalities hold for every $h\in \hh$ with
$\|h\| < \frac{1}{2} \epsilon$,
   \begin{align*}
\is{L_{m}h}{h} = \frac{1}{2} \sum_{j=0}^2 (-1)^{j}
\binom{2}{j} \is{L_{m}(h_0+ j h)}{h_0+ j
h}\overset{\eqref{miso-a}}=0.
   \end{align*}
By the homogeneity of $L_{m}$, this implies that
$\is{L_{m}h}{h}=0$ for all $h\in \hh$, or equivalently
that $L_{m}=0$. This completes the proof.
   \end{proof}
   The third lemma provides yet another
characterization of $m$-isometries.
   \begin{lem} \label{mialo-b}
If $m\in \nbb$, then an operator $T\in \ogr{\hh}$ is
an $m$-isometry if and only if there exists a nonempty
open subset $\mathscr V$ of $\hh$ such that for every
$h\in \mathscr V$, $\|T^nh\|^2$ is a polynomial in $n$
of degree at most $m-1$.
   \end{lem}
   \begin{proof}
The ``only if'' part follows from
Theorem~\ref{Newt4}(iii). To prove the ``if'' part,
take $h\in \mathscr V$. Since by assumption
$\|T^nh\|^2$ is a polynomial in $n$ of degree at most
$m-1$, we infer from Proposition~\ref{kerdelta} that
$\gammab_{T,h} \in \ker \triangle^{m}$. Hence, by
Proposition~\ref{Newt3}(ii), we have
   \begin{align*}
(-1)^{m} \is{\bscr_{m}(T)h}{h} = (\triangle^{m}
\gammab_{T,h})_0 = 0.
   \end{align*}
Hence $\is{\bscr_{m}(T)h}{h}=0$ for every $h\in
\mathscr V$. By Lemma~\ref{lnzero} or directly by
Lemma~\ref{polarf}, $\bscr_{m}(T)=0$, which means that
$T$ is an $m$-isometry.
   \end{proof}
   We are now ready to state and prove ``local''
characterizations of $m$-isometries. We also give a
topological description of certain sets of vectors $h$
having the property that $\|T^n h\|^2$ is a polynomial in
$n$ of prescribed degree, where $T$ is a strict
$m$-isometry. Note that the implication
(ii)$\Rightarrow$(i) of Theorem~\ref{weakmiso} below is not
true if we drop the assumption that $\hh$ is complete (see
Remark~\ref{wiel-nod}). Recall that $\cscr_T(h)$ is the
linear span of $\{T^n h\colon n\in \zbb_+\}$.
   \begin{thm} \label{weakmiso}
The following conditions are equivalent for $T\in
\ogr{\hh}${\em :}
   \begin{enumerate}
   \item[(i)] there exists $m\in \nbb$ such that $T$
is an $m$-isometry,
   \item[(ii)] for any  $h\in \hh$,
$\|T^nh\|^2$ is a polynomial in $n$,
   \item[(iii)] $\ker T=\{0\}$ and for any $h\in
\hh\setminus \{0\}$, there exists $m_h\in \nbb$ such
that the unilateral weighted shift $W_{T,h}$ with
weights $\big\{\frac{\|T^{n+1}h\|} {\|T^n
h\|}\big\}_{n=0}^{\infty}$ is an
\mbox{$m_h$-isometry},
   \item[(iv)] there exists a nonempty open subset $\mathscr V$
of $\hh$ such that for any $h\in \mathscr V$, there
exists $m_h\in\nbb$ such that
   \begin{align*}
\is{\bscr_{m_h}(T)h}{h}=0,
   \end{align*}
   \item[(v)] for any $h\in \hh$, there exists $m_h\in\nbb$
such that $T|_{\overline{\cscr_T(h)}}$ is an
$m_h$-isometry.
   \end{enumerate}
Moreover, if $T\in \ogr{\hh}$ is a strict
$m$-isometry, where $m \Ge 2$, then
   \begin{enumerate}
   \item[(a)] $\hh=\Big\{h \in \hh\colon \|T^nh\|^2 \text{ is a
polynomial in $n$ of degree at most $m-1$}\Big\}$,
   \item[(b)]
$\fscr:=\Big\{h \in \hh\colon \|T^nh\|^2 \text{ is a
polynomial in $n$ of degree at most $m-2$}\Big\}$ is a
closed nowhere dense subset of $\hh$,
   \item[(c)]
$\mathscr U:=\Big\{h \in \hh\colon \|T^nh\|^2 \text{
is a polynomial in $n$ of degree $m-1$}\Big\}$ is an
open dense subset of $\hh$.
   \end{enumerate}
   \end{thm}
   \begin{proof}
The implications (i)$\Rightarrow$(v) and
(v)$\Rightarrow$(iv) are obvious due to
\eqref{miso-1}. The implication (iv)$\Rightarrow$(i)
follows from Lemma~\ref{lnzero}, while the implication
(i)$\Rightarrow$(iii) is a direct consequences of
Theorem~\ref{Newt4}(iv).

(iii)$\Rightarrow$(ii) Assume that (iii) holds. Fix a
nonzero $h\in \hh$. By assumption, there is $m_h\in \nbb$
such that $W_{T,h}$ is an $m_h$-isometry. Let
$e_0=(1,0,0,\ldots)$ be the zeroth term of the standard
orthonormal bass of $\ell^2$. By Theorem~\ref{Newt4}(iii),
$\|W_{T,h}^ne_0\|^2$ is a polynomial in $n$ of degree at
most $m_h-1$; denote it by $p$. Then
   \begin{align*}
p(n)=\|W_{T,h}^ne_0\|^2 = \frac{\|T^n h\|^2}{\|h\|^2},
\quad n\in \zbb_+,
   \end{align*}
which yields (ii).

We now prove that (ii) implies (iv). For this, assume
that (ii) holds. Take $h\in \hh\setminus \{0\}$. Then
by assumption $\|T^nh\|^2$ is a polynomial in $n$ of
some degree $k_h\in \zbb_+$. By
Proposition~\ref{kerdelta}, $\gammab_{T,h} \in \ker
\triangle^{m_h}$ with $m_h=k_h+1$. Hence, by
Proposition~\ref{Newt3}(ii), we have
   \begin{align*}
(-1)^{m_h} \is{\bscr_{m_h}(T)h}{h} = (\triangle^{m_h}
\gammab_{T,h})_0 = 0.
   \end{align*}
Therefore (iv) is valid. This means that the
conditions (i)-(v) are equivalent.

It remains to prove the moreover part. The condition
(a) is a direct consequence of
Theorem~\ref{Newt4}(iii). To prove (b), recall that
$\cbb^{\zbb_+}$ is a topological vector space with the
topology of pointwise convergence on $\zbb_+$ denoted
here by $\tau$. Write $\mscr$ for the vector subspace
of $\cbb^{\zbb_+}$ which consists of all polynomials
in $n$ of degree at most $m-2$. We first show that
$\fscr$ is a closed subset $\hh$. Indeed, if
$\{h_k\}_{k=1}^{\infty} \subseteq \fscr$ converges to
$h\in \hh$, then by the continuity of $T$, the
sequence $\{\gammab_{T,h_k}\}_{k=1}^{\infty} \subseteq
\mscr$ is $\tau$-convergent to $\gammab_{T,h}$ in
$\cbb^{\zbb_+}$. Since any finite dimensional vector
subspace of a topological vector space is closed, we
deduce that $\gammab_{T,h} \in \mscr$, which means
that $h\in \fscr$. Hence $\fscr$ is closed. Suppose,
to the contrary, that $\fscr$ is not a nowhere dense
subset of $\hh$. Then, by Lemma~\ref{mialo-b}, $T$ is
an $(m-1)$-isometry, which contradicts our assumption
that $T$ is a strict $m$-isometry. Finally, since by
(a), $\mathscr U=\hh\setminus \fscr$, the condition
(c) follows from (b). This completes the proof.
   \end{proof}
   \begin{rem}
Regarding the moreover part of Theorem~\ref{weakmiso},
it is worth pointing out that for any integer $m\Ge
2$, there exists a strict $m$-isometry. This was shown
by Athavale in the case of infinite dimensional
separable Hilbert spaces (see
\cite[Proposition~8]{At91}). In turn, in view of
Proposition~\ref{m-iso-alg} below, for any positive
odd number $m$, there exists a strict $m$-isometry
which is algebraic; what is more, this may happen even
in finite dimensional Hilbert spaces (see
Remark~\ref{inv-miso}). This justifies the following
construction. Take integers $m_1$ and $m_2$ such that
$2 \Le m_1 < m_2$. Let for $j=1,2$, $T_j\in
\ogr{\hh_j}$ be a strict $m_j$-isometry. Set
$\hh=\hh_1 \oplus \hh_2$ and $T=T_1 \oplus T_2$. Using
Theorem~\ref{Newt4}(iii) we deduce that $T$ is a
strict $m_2$-isometry. It is clear that for every $h
\in \hh_1 \oplus \{0\}$, $\|T^nh\|^2$ is a polynomial
in $n$ of degree at most $m_1-1$, hence of degree at
most $m_2-2$. In view of Theorem~\ref{weakmiso}, the
set $\mathscr U$ is open and dense in $\hh$ and its
complement $\fscr$ contains the closed vector subspace
$\hh_1 \oplus \{0\}$. \hfill{$\diamondsuit$}
   \end{rem}
   \section{\label{Sec4}$m$-isometries on inner product spaces}
Motivated by our investigations and needs of the
theory of unbounded operators (see e.g.,
\cite{Ja-St01}), we consider here $m$-isometries on
inner product spaces. We do not assume that the
operators in question are continuous. Suppose $\mscr$
is an inner product space and $\hh$ is its Hilbert
space completion. Denote by $\lino{\mscr}$ the algebra
of all linear operators on $\mscr$. A member of
$\lino{\mscr}$ can be though of as a densely defined
operator in $\hh$ with invariant domain $\mscr$. Given
$T\in \lino{\mscr}$ and $m\in \nbb$, we say that $T$
is an {\em $m$-isometry}\/\footnote{\;Clearly, if
$\mscr=\hh$ and $T\in \ogr{\hh}$, then both
definitions of $m$-isometricity, the present one and
that of Section~\ref{Sec1}, coincide (cf.\
\eqref{miso-1}).} if
   \begin{align} \label{def-miso}
\hat\fcal_{T;m}(h):=\sum\limits_{k=0}^m (-1)^k
\binom{m}{k} \|T^k h\|^2 = 0, \quad h \in \mscr.
   \end{align}
Similarly to Section \ref{Sec1}, we define the notion
of a strict $m$-isometry. Arguing as in
Section~\ref{Sec2}, we verify that $T$ is an
$m$-isometry if and only if for each $h\in \mscr$,
$\|T^nh\|^2$ is a polynomial in $n$ of degree at most
$m-1$ (cf.\ Theorem~\ref{Newt4}). Consequently, if $T$
is an $m$-isometry, then $T$ is a $k$-isometry for
every integer $k \Ge m$. If $\mscr$ is infinite
dimensional it may happen that an $m$-isometry on
$\mscr$ is closed as an operator in $\hh$ but not
bounded (see e.g.\ \cite[Example~6.4]{Ja-St01}).

We begin our investigations by making the following
observation.
   \begin{pro} \label{polar-1}
If $T\in\lino{\mscr}$ is an $m$-isometry, then for all
$f,g\in \mscr$, $\is{T^n f}{T^n g}$ is a polynomial in
$n$ of degree at most $m-1$, and
   \begin{align*}
\is{T^n f}{T^n g}= \sum_{k=0}^{m-1} (n)_k
\frac{(-1)^k}{k!} \fcal_{T;k}(f,g), \quad n\in \zbb_+,
\, f, g \in \mscr,
   \end{align*}
where $\fcal_{T;k}$ is the sesquilinear form on
$\mscr$ defined by
   \begin{align*}
\fcal_{T;k}(f,g) = \sum_{j=0}^k (-1)^j \binom{k}{j}
\is{T^j f}{T^j g}, \quad f,g \in \mscr, \, k \in
\zbb_+.
   \end{align*}
   \end{pro}
   \begin{proof}
We can argue as in Section~\ref{Sec2} and apply the
polarization formula to the sesquilinear form
$\mscr\times\mscr \ni (f,g) \to \is{T^n f}{T^n g} \in
\cbb$.
   \end{proof}
Recall that an operator $N \in \lino{\mscr}$ is said
to be a {\em nilpotent operator} if there exists $k\in
\nbb$ such that $N^k=0$. If additionally $N^{k-1}\neq
0$, then $k$ is called the {\em index of nilpotency}
of $N$ and is denoted here by $\nil{N}$; note that
then linear dimension of $\mscr$ must be at least
$\nil{N}$, so in particular $\mscr\neq \{0\}$. Again,
as in the case of $m$-isometries, if $\mscr$ is
infinite dimensional it may happen that a nilpotent
operator on $\mscr$ is closed as an operator in $\hh$
but not bounded (see \cite[Example~3.3]{Ota88}).

Theorem~\ref{Bang1-m}(i) below was proved by
Berm\'{u}dez, Martin\'{o}n, M\"{u}ller and Noda in the
case of bounded Hilbert space operators (see
\cite[Theorem~3]{Ber-Mar-Mul-No14}). In turn,
Theorem~\ref{Bang1-m}(ii) was shown by Le and
independently by Gu and Stankus, again for Hilbert
space operators (see \cite[Theorem~16]{Le15} and
\cite[Theorem~4]{Gu-St15}, resp.). The proof of
Theorem~\ref{Bang1-m}(i) is an adaptation (and a
simplification) of that of
\cite[Theorem~3]{Ber-Mar-Mul-No14} to the context of
inner product spaces.
   \begin{thm} \label{Bang1-m}
Suppose $m\in \nbb$, $A\in \lino{\mscr}$ is an
$m$-isometry and $N\in \lino{\mscr}$ is a nilpotent
operator such that $AN=NA$. Set $m_N=m +
2(\nil{N}-1)$. Then
   \begin{enumerate}
   \item[(i)] $A+N$ is an $m_N$-isometry,
   \item[(ii)] $A+N$ is a strict $m_N$-isometry
if and only if there exists $f_0\in \mscr$ such~that
   \begin{align*}
\sum_{l=0}^{m-1} (-1)^{l} \binom{m-1}{l} \|A^{l}
N^{\nil{N}-1} f_0\|^2 \neq 0.
   \end{align*}
   \end{enumerate}
Moreover, if $A+N$ is a strict $m_N$-isometry, then
$A$ is a strict $m$-isometry.
   \end{thm}
   \begin{proof}
(i) Set $k=\nil{N}$ and $\kappa_n=\min\{n,k-1\}\in \zbb_+$
for $n\in \zbb_+$. By assumption and
Proposition~\ref{polar-1}, there exist polynomials
$p_{i,j;f}, q_{i,j;f} \in \cbb[x]$ such~that
   \allowdisplaybreaks
   \begin{align} \label{iso-21}
\left\langle A^{n}\Big(A^{j-i}N^{i}f\Big),
A^{n}\Big(N^{j}f\Big)\right\rangle = p_{i,j;f}(n),
\quad 0\Le i< j, \, n\in \zbb_+, \, f\in \mscr,
   \\ \label{iso-22}
\left\langle A^{n}\Big(N^{i}f\Big),
A^{n}\Big(A^{i-j}N^{j}f\Big)\right\rangle =
q_{i,j;f}(n), \quad 0\Le j\Le i, n\in \zbb_+, \, f\in
\mscr.
   \end{align}
Using Newton's binomial formula, we see that
   \allowdisplaybreaks
   \begin{align}  \notag
\|(A+N)^n f\|^2 & = \left\langle \sum_{i=0}^{\kappa_n}
\frac{(n)_i}{i!} A^{n-i}N^if, \sum_{j=0}^{\kappa_n}
\frac{(n)_j}{j!} A^{n-j}N^jf\right\rangle
   \\ \notag
& = \sum_{0\Le i < j \Le \kappa_n} \frac{(n)_i}{i!}
\frac{(n)_j}{j!} \left\langle
A^{n-j}\Big(A^{j-i}N^{i}f\Big),
A^{n-j}\Big(N^{j}f\Big)\right\rangle
   \\  \notag
& \hspace{4ex}+ \sum_{0\Le j \Le i \Le \kappa_n}
\frac{(n)_i}{i!} \frac{(n)_j}{j!} \left\langle
A^{n-i}\Big(N^{i}f\Big),
A^{n-i}\Big(A^{i-j}N^{j}f\Big)\right\rangle
   \\ \notag
& \hspace{-3.65ex}\overset{\eqref{iso-21} \&
\eqref{iso-22}}= \sum_{0\Le i < j \Le k-1}
\frac{(n)_i}{i!} \frac{(n)_j}{j!} p_{i,j;f}(n-j)
   \\ \label{zeg-1}
& \hspace{8ex} + \sum_{0\Le j \Le i \Le k-1}
\frac{(n)_i}{i!} \frac{(n)_j}{j!} q_{i,j;f}(n-i),
\quad n \in \zbb_+, f\in \mscr.
   \end{align}
Since $(n)_l$ is a polynomial in $n$ of degree $l$
and, by Proposition~\ref{polar-1}, $p_{i,j;f},
q_{i,j;f}$ are polynomials of degree at most $m-1$, we
conclude that $\|(A+N)^n f\|^2$ is a polynomial in $n$
of degree at most $m - 1 + 2(k-1)$. Therefore, $A+N$
is $m_N$-isometry.

(ii) Observe that by Proposition~\ref{polar-1}, the
coefficient of the polynomial $q_{k-1,k-1;f}$ that
corresponds to the monomial $x^{m-1}$ equals
$\frac{(-1)^{m-1}}{(m-1)!}\fcal_{A;m-1}(N^{k-1} f,
N^{k-1}f)$. This together with \eqref{zeg-1} yields (ii).

The ``moreover'' part is a consequence of (ii). This
completes the proof.
   \end{proof}
Regarding the moreover part of Theorem~\ref{Bang1-m},
it is worth pointing out that $A+N$ may not be a
strict $m_N$-isometry even if $A$ is a strict
$m$-isometry (see the paragraph preceding
\cite[Theorem 3]{Ber-Mar-Mul-No14}).

The question of when a scalar translate of a nilpotent
operator is an $m$-isometry is answered by
Proposition~\ref{Bang1} below. The assertion (i) of
this proposition was proved by Berm\'{u}dez,
Martin\'{o}n and Noda in the case of bounded Hilbert
space operators (see
\cite[Theorem~2.2]{Ber-Mar-No13}). Before proving
Proposition~\ref{Bang1}, we state a simple lemma whose
proof is left to the reader.
   \begin{lem} \label{lim-1n}
If $p\in \cbb[x]$ is a nonzero polynomial, then
$\lim_{n\to\infty} |p(n)|^{1/n}=1$.
   \end{lem}
%   \begin{proof}
%Set $k=\deg{p} \in \zbb_+$. Let $\lambda$ be the
%leading coefficient of $p$. Then clearly
%$\lim_{n\to\infty} \frac{|p(n)|}{n^{k}}=|\lambda| > 0$
%and so $\lim_{n\to\infty}
%\big(\frac{|p(n)|}{n^{k}}\big)^{1/n}=1$, which gives
%   \begin{align*}
%\lim_{n\to\infty} |p(n)|^{1/n} = \lim_{n\to\infty}
%(n^{1/n})^k \bigg(\frac{|p(n)|}{n^{k}}\bigg)^{1/n}=1.
%   \end{align*}
%   \end{proof}
   \begin{pro} \label{Bang1}
Suppose $\mscr\neq \{0\}$ and $N\in \lino{\mscr}$ is a
nilpotent operator. Then the following statements are
valid{\em :}
   \begin{enumerate}
   \item[(i)] if  $V\in \lino{\mscr}$ is an isometry such that
$VN=NV$, then $V+N$ is a strict $(2 \nil{N}-1)$-isometry,
   \item[(ii)] if $z\in \cbb$ is such that
$\|(zI+N)^n h\|^2$ is a polynomial in $n$ for every
$h\in \mscr$, then $|z|=1$,
   \item[(iii)] if $z\in \cbb$ and
$zI+N$ is an $m$-isometry, then~$|z|=1$.
   \end{enumerate}
   \end{pro}
   \begin{proof}
In view of Proposition~\ref{polar-1} and
Theorem~\ref{Bang1-m}, it suffices to prove (ii). Under the
assumptions of (ii), we argue as follows. Set $k=\nil{N}$.
Let $h_0\in \mscr$ be such that $N^{k-1}h_0\neq 0$ and let
$p\in \cbb[x]$ be a polynomial such that
   \begin{align*}
p(n)=\|(zI+N)^n h_0\|^2, \quad n\in \zbb_+.
   \end{align*}
Clearly $p\neq 0$. Consider first the case $z=0$. Then
$p(n)=\|N^n h_0\|^2=0$ for all $n\Ge k$, which means
that the polynomial $p$ has infinitely many roots.
Hence $p=0$, a contradiction. Suppose now that $z\neq
0$. Using Newton's binomial formula, we obtain
   \begin{align} \label{slon-2}
p(n) = \bigg\|\sum_{j=0}^{k-1} \frac{(n)_j}{j!}
z^{n-j} N^j h_0\bigg\|^2 = |z|^{2n}
\bigg\|\sum_{j=0}^{k-1} \frac{(n)_j}{j!z^j} N^j
h_0\bigg\|^2, \quad n\in \zbb_+.
   \end{align}
Notice that $\big\|\sum_{j=0}^{k-1}
\frac{(n)_j}{j!z^j} N^j h_0\big\|^2$ is a polynomial
in $n$ which is nonzero (because $p(0)\neq 0$).
Combined with Lemma~\ref{lim-1n}, this yields
   \begin{align*}
1=\lim_{n\to\infty} p(n)^{1/n}
\overset{\eqref{slon-2}}= |z|^2,
   \end{align*}
which completes the proof.
   \end{proof}
It is worth noting that the proof of
Proposition~\ref{Bang1}(iii) can be made much shorter
if we assume additionally that $\mscr=\hh$ and
$N\in\ogr{\hh}$. Indeed, then by the spectral mapping
theorem, $r(zI+N)=|z|$. Therefore, if $zI+N$ is an
$m$-isometry, then by \eqref{as-spec}, $|z|=1$.
   \section{\label{Sec5}Orthogonality of generalized eigenvectors}
Theorem~\ref{alg-m-iso-lem} below provides a few
sufficient conditions for orthogonality of generalized
eigenvectors corresponding to distinct eigenvalues.
First we need the following lemma of algebraic nature
whose proof depends heavily on the Carlson theorem
recorded below.
    \begin{thm}[\mbox{\cite{Car},\cite[Theorem 1]{Rubel}}] \label{Ca-Ru}
Let $\phi$ be an entire function on $\cbb$ such that
$\sup_{z \in \cbb} |\phi(z)|\E^{-\tau |z|} < \infty$
for some $\tau\in \rbb$ and $\sup_{y \in \rbb}
|\phi(\I y)|\E^{-\eta |y|} < \infty$ for some $\eta\in
(-\infty,\pi)$. If $\phi(n)=0$ for every $n\in\nbb$,
then $\phi(z)=0$ for all $z\in \cbb$.
    \end{thm}
   \begin{lem} \label{re-polyn}
Suppose $p,q \in \cbb[x]$ and $\alpha \in \tbb \setminus
\{1\}$ are such that
   \begin{align} \label{penre}
p(n) = \mathrm{Re} (\alpha^n q(n)), \quad n \in
\zbb_+.
   \end{align}
Then $\mathrm{Re}\, q = 0$ if $\alpha=-1$, and $q=0$
otherwise.
   \end{lem}
   \begin{proof}
We consider three possible cases that are logically
disjoint.

{\sc Case 1.} $\alpha=-1$.

It follows from \eqref{penre} that
   \begin{align*}
p(2n) = (\mathrm{Re} \, q)(2n), \quad n\in \zbb_+.
   \end{align*}
Then, by the Fundamental Theorem of Algebra, we have
   \begin{align} \label{penre2}
p (n) = (\mathrm{Re} \, q)(n), \quad n\in \zbb_+.
   \end{align}
As a consequence, we obtain
   \begin{align*}
(\mathrm{Re} \, q)(2n+1) \overset{\eqref{penre2}}=
p(2n+1) \overset{\eqref{penre}}= - (\mathrm{Re} \,
q)(2n+1), \quad n\in \zbb_+.
   \end{align*}
Thus, by the Fundamental Theorem of Algebra again,
$\mathrm{Re} \, q = 0$.

{\sc Case 2.} $\alpha=\E^{\I \vartheta}$ for some
$\vartheta \in (0,\pi)$.

Observe that
   \begin{align} \notag
p(n) & \overset{\eqref{penre}}= \mathrm{Re} \, (\E^{\I
n \vartheta}) \mathrm{Re} \,(q(n)) - \mathrm{Im} \,
(\E^{\I n \vartheta}) \mathrm{Im} \,(q(n))
   \\ \label{penre3}
&\hspace{1ex} =\cos(n \vartheta) (\mathrm{Re} \,q)(n)
- \sin (n \vartheta) (\mathrm{Im} \,q)(n), \quad n\in
\zbb_+.
   \end{align}
Define the entire function $\phi$ on $\cbb$ by
   \begin{align}  \label{penre2.5}
\phi(z) = p(z) - \cos(\vartheta z) (\mathrm{Re}
\,q)(z) + \sin (\vartheta z) (\mathrm{Im} \,q)(z),
\quad z\in \cbb.
   \end{align}
It follows from \eqref{penre3} that
   \begin{align} \label{penre4}
\phi(n) = 0, \quad n\in \zbb_+.
   \end{align}
Now we show that for any $\epsilon \in (0,\infty)$,
there exists $c_{\epsilon} \in (0,\infty)$ such that
   \begin{align} \label{penre5}
|\phi(z)| \Le c_{\epsilon} \E^{(\vartheta + \epsilon)
|z|}, \quad z \in \cbb.
   \end{align}
Indeed, if $r \in \cbb[x]$, then for every $\epsilon
\in (0,\infty)$ there exists $d_{\epsilon} \in
(0,\infty)$ such that $|r(z)| \Le d_{\epsilon}
\E^{\epsilon |z|}$ for all $z\in \cbb$. It is easily
seen that $|\cos(z\vartheta)| \Le \E^{\vartheta |z|}$
and $|\sin(z\vartheta)| \Le \E^{\vartheta |z|}$ for
all $z\in \cbb$. Putting this all together yields
\eqref{penre5}.

By \eqref{penre5}, $\phi$ is of exponential type
(i.e., $|\phi(z)| \Le c \E^{\tau |z|}$ for all $z \in
\cbb$ and for some $c,\tau\in \rbb$) and there exist
$d\in (0,\infty)$ and $\eta \in (0,\pi)$ such~that
   \begin{align*}
|\phi(\I y)| \Le d\E^{\eta |y|}, \quad y\in \rbb.
   \end{align*}
These two facts and \eqref{penre4} imply that the
entire function $\phi$ satisfies the assumptions of
Theorem \ref{Ca-Ru}. Therefore by this theorem
$\phi=0$. Combined with \eqref{penre2.5}, this yields
   \begin{align} \label{penre6}
p(z) = \cos(\vartheta z) (\mathrm{Re} \,q)(z) - \sin
(\vartheta z) (\mathrm{Im} \,q)(z), \quad z\in \cbb.
   \end{align}
Substituting $z=\I y$ with $y\in \rbb$ into
\eqref{penre6}, we obtain
   \begin{align} \label{penre7}
p(\I y) = u_1(y) \E^{-\vartheta y} + u_2(y)
\E^{\vartheta y}, \quad y\in \rbb,
   \end{align}
where $u_1, u_2 \in \cbb[x]$ are polynomials given by
   \begin{align}   \label{penre7.5}
2 u_1 = (\mathrm{Re} \,q)(\I x) + \I (\mathrm{Im}
\,q)(\I x) \quad \text{and} \quad 2 u_2 = (\mathrm{Re}
\,q)(\I x) - \I (\mathrm{Im} \,q)(\I x).
   \end{align}
It follows from \eqref{penre7} that
   \begin{align} \label{penre8}
u_2(y)=\frac{p(\I y)}{\E^{\vartheta y}} -
\frac{u_1(y)}{\E^{2 \vartheta y}}, \quad y\in \rbb.
   \end{align}
Using the fact that $p, u_1, u_2 \in \cbb[x]$ and passing
to the limit as $y\to \infty$ in \eqref{penre8}, we deduce
that $u_2=0$. A similar argument shows that $u_1=0$.
Combined with \eqref{penre7.5}, this implies that
$\mathrm{Re} \,q = 0$ and $\mathrm{Im} \,q = 0$. Hence
$q=0$.

{\sc Case 3.} $\alpha=\E^{\I \vartheta}$ for some
$\vartheta \in (\pi,2\pi)$.

In view of \eqref{penre} we have
   \begin{align*}
p(n) = \mathrm{Re} (\bar \alpha^n q^*(n)), \quad n \in
\zbb_+.
   \end{align*}
Since $\bar \alpha = \E^{\I (2\pi -\vartheta)}$ and
$2\pi -\vartheta \in (0,\pi)$, we can apply Case 2. As
a consequence, we obtain $q^*=0$. Thus $q=0$, which
completes the proof.
   \end{proof}
For the reader's convenience, we recall necessary
terminology related to generalized eigenvectors. Given
$T\in\lino{\mscr}$ and $z\in \cbb$, we set
   \begin{align*}
\gev{T}{z}=\bigcup_{n\in \nbb} \ker((T-z I)^{n}).
   \end{align*}
Since $\ker((T-z I)^{n}) \subseteq \ker((T-z
I)^{n+1})$ for all $n\in \nbb$, we deduce that
$\gev{T}{z}$ is a vector subspace of $\mscr$ which is
invariant for $T$. Notice that $\gev{T}{z} \neq \{0\}$
if and only if $z$ is an eigenvalue of $T$. If $z$ is
an eigenvalue of $T$, then a nonzero element of
$\gev{T}{z}$ is called a {\em generalized eigenvector}
of $T$ corresponding to $z$, and the vector space
$\gev{T}{z}$ itself is called the {\em generalized
eigenspace} (or {\em spectral space}) of $T$
corresponding to $z$.

We are now ready to prove the aforesaid criterions for
orthogonality of generalized eigenvectors. Recall that
$\ubb=\{-1,1\} \times \{-\I,\I\}$.
   \begin{thm} \label{alg-m-iso-lem}
Let $z_1$ and $z_2$ be two distinct elements of
$\tbb$. Suppose $T\in\lino{\mscr}$ and $h_j \in
\gev{T}{z_j}$ for $j=1,2$. Then the following
assertions hold{\em :}
   \begin{enumerate}
   \item[(i)] $\mathrm{Re}\is{T^nh_1}{T^nh_2}=0$ for all
$n\in \zbb_+$ provided $z_1=-z_2$ and $\|T^n(h_1 +
h_2)\|^2$ is a polynomial in $n$,
   \item[(ii)] $\is{T^nh_1}{T^nh_2}=0$
for all $n\in \zbb_+$ provided $z_1=-z_2$ and there is
$(\epsilon_1,\epsilon_2) \in\ubb$ such that
$\|T^n(\epsilon_k h_1 + h_2)\|^2$ is a polynomial in
$n$ for $k=1,2$,
   \item[(iii)] $\is{T^nh_1}{T^nh_2}=0$
for all $n\in \zbb_+$ provided $z_1\neq -z_2$ and
$\|T^n(h_1 + h_2)\|^2$ is a polynomial in $n$.
   \end{enumerate}
   \end{thm}
   \begin{proof}  By assumption there exit $k_1,k_2 \in \nbb$ such that
$h_j \in \ker((T - z_j I)^{k_j})$ for $j=1,2$. Before
justifying the assertions (i)-(iii), we discuss the general
case when $\|T^n(h_1 + h_2)\|^2$ is a polynomial in $n$.
Let $r\in \cbb[x]$ be a polynomial such that
$r(n)=\|T^n(h_1 + h_2)\|^2$ for all $n\in \zbb_+$. Set
$N_j= T - z_j I$ for $j=1,2$. Then
   \begin{align} \notag
r(n) & = \|(I+ \bar z_1 N_1)^n h_1\|^2 + \|(I+ \bar
z_2 N_2)^n h_2\|^2
   \\ \label{pietrz1}
&\hspace{6.5ex}+ 2 \mathrm{Re} (\alpha^n \is{(I + \bar
z_1 N_1)^n h_1}{(I + \bar z_2 N_2)^n h_2}), \quad n\in
\zbb_+,
   \end{align}
where $\alpha:=z_1 \bar z_2$. Since $z_1, z_2 \in
\tbb$ and $z_1 \neq z_2$, we see that $\alpha \in
\tbb\setminus \{1\}$. By assumption and Newton's
binomial formula, we have
   \begin{align}  \label{pietrz3}
(I+ \bar z_j N_j)^n h_j = \sum_{l=0}^n \binom{n}{l}
\bar z_j^l N_j^l h_j = \sum_{l=0}^{k_j-1}
\frac{(n)_l}{l!} \bar z_j^l N_j^l h_j, \quad j=1,2.
   \end{align}
Since for any $l\in \zbb_+$, $(n)_l$ is a polynomial in
$n$, \eqref{pietrz3} implies that $\|(I + \bar z_1 N_1)^n
h_1\|^2$, $\|(I + \bar z_2 N_2)^n h_2\|^2$ and $\is{(I+
\bar z_1 N_1)^n h_1}{(I+ \bar z_2 N_2)^n h_2}$ are
polynomials in $n$. It follows form \eqref{pietrz1} that
there exist polynomials $p, q \in \cbb[x]$ such that
   \begin{align} \label{pietrz2}
p(n) = \mathrm{Re}(\alpha^n q(n)), \quad n\in \zbb_+,
   \end{align}
   where
   \begin{align*}
q(n) = \is{(I+ \bar z_1 N_1)^n h_1}{(I+ \bar z_2
N_2)^n h_2}, \quad n\in \zbb_+.
   \end{align*}
The above discussion also gives the following identity
   \begin{align} \label{asu-1}
\is{T^n h_1}{T^n h_2} = \alpha^n q(n), \quad n\in
\zbb_+.
   \end{align}

(i) If $z_1=-z_2$, then $\alpha=-1$ and thus by
\eqref{pietrz2} and Lemma~\ref{re-polyn},
$\mathrm{Re}\, q = 0$, or equivalently by
\eqref{asu-1}, $\mathrm{Re}\is{T^nh_1}{T^nh_2}=0$ for
all $n\in \zbb_+$.

(ii) Assume that $\|T^n(\epsilon h_1 + h_2)\|^2$ is a
polynomial in $n$ for $\epsilon\in \{1,\I\}$ (the
remaining cases can be proved in the same way).
Applying (i) to the pairs $(h_1,h_2)$ and $(\I
h_1,h_2)$, we see that
   \begin{align*}
\mathrm{Re}\is{T^nh_1}{T^nh_2}=0=\mathrm{Im}\is{T^nh_1}{T^nh_2},
\quad n \in \zbb_+,
   \end{align*}
which yields (ii).

(iii) Since now $\alpha \in \tbb\setminus \{-1\}$, the
assertion (iii) is a direct consequence of
\eqref{pietrz2}, \eqref{asu-1} and
Lemma~\ref{re-polyn}. This completes the proof.
   \end{proof}
   \begin{cor} \label{Bang-ch}
Let $z_1$ and $z_2$ be two distinct elements of
$\tbb$. Suppose $T\in\lino{\mscr}$ and $h_j \in
\gev{T}{z_j}$ for $j=1,2$. Consider the following
conditions{\em :}
   \begin{enumerate}
   \item[(i)] $z_1 = -z_2$ and there is
$(\epsilon_1,\epsilon_2) \in \ubb$ such that
$\|T^n(\epsilon_k T^j h_1 + h_2)\|^2$ is a polynomial
in $n$ for $k=1,2$ and every $j\in \zbb_+$,
   \item[(ii)] $z_1 \neq -z_2$ and
$\|T^n(T^j h_1 + h_2)\|^2$ is a polynomial in $n$ for
every $j\in \zbb_+$.
   \end{enumerate}
Then any of the conditions {\em (i)} and {\em (ii)} implies
that $\overline{\cscr_T(h_1)} \perp
\overline{\cscr_T(h_2)}$.
   \end{cor}
Regarding Theorem~\ref{alg-m-iso-lem} and its proof,
it is worth noticing that if $z\in \tbb$ is an
eigenvalue of an operator $T\in \lino{\mscr}$ and $h$
is in $\gev Tz$, then $T^n h$ may not be a polynomial
in $n$ (though, by Proposition~\ref{Bang1}, $\|T^n
h\|^2$ is). A simple example is given below.
   \begin{exa}
Take $\mscr=\cbb^2$, $z\in \tbb\setminus \{1\}$, $T=
[\begin{smallmatrix} z & 1
\\ 0 & z \end{smallmatrix}]$ and
$h=[\begin{smallmatrix} 1 \\ 0 \end{smallmatrix}]$.
Then $h\in \ker(T-zI)$ and $T^n h = z^n h$ for all
$n\in \zbb_+$. As a consequence, $\|T^n h\|^2=\|h\|^2$
is a polynomial in $n$, however $T^n h$ is not a
polynomial in $n$ because the sequence $\{T^n
h\}_{n=0}^{\infty}$ is not constant and nonconstant
polynomials in $n$ are unbounded.
\hfill{$\diamondsuit$}
   \end{exa}
It turns out that the assumptions of
Theorem~\ref{alg-m-iso-lem}(i) do not imply that $h_1
\perp h_2$. In fact, it can be even worse as shown in
an example below.
   \begin{exa}
Let $\mscr=\cbb^2$, $T=[\begin{smallmatrix} \I & 2
\\ 0 & -\I \end{smallmatrix}]$, $h_1=[\begin{smallmatrix} 1
\\ 0 \end{smallmatrix}]$,  $h_2=[\begin{smallmatrix} \I
\\ 1 \end{smallmatrix}]$, $z_1=\I$ and $z_2=-\I$. Then
$z_1,z_2 \in \tbb$, $z_1\neq z_2$, $z_1 = - z_2$, $h_1
\in \ker(T-z_1 I)$, $h_2 \in \ker(T-z_2 I)$ and
$T^2=-I$. It is now a matter of routine to verify that
$\|T^n(h_1+h_2)\|^2=3$ for all $n\in \zbb_+$, which
means that $\|T^n(h_1+h_2)\|^2$ is a polynomial in
$n$. Clearly $\is{T^kh_1}{T^lh_2}=-\I^{k+l+1}$ for all
$k,l\in \zbb_+$. \hfill{$\diamondsuit$}
   \end{exa}
   \section{\label{Sec6}Jordan blocks and algebraic operators}
We begin by recalling the necessary terminology and facts
related to the concept of Jordan's block. Let $\mscr$ be an
inner product space and $\hh$ be its Hilbert space
completion. Suppose that $z$ is an eigenvalue of $T\in
\lino{\mscr}$ and $h$ is in $\gev{T}{z}\setminus \{0\}$.
Then there exists $k \in \nbb$ such that
$h\in\ker((T-zI)^k)$. Using Newton's binomial formula, we
deduce that $\cscr_T(h)$ is the linear span of the vectors
$h$, $Th$, \ldots, $T^{k-1}h$, so $\cscr_T(h)$ is a finite
dimensional subspace of $\ker((T-zI)^k) \subseteq
\gev{T}{z}$. Clearly, we have
   \begin{align} \label{cyc2}
(T|_{\cscr_T(h)} - z I_{\cscr_T(h)})^{k}=(T - z
I)^{k}|_{\cscr_T(h)}=0,
   \end{align}
which implies that $T|_{\cscr_T(h)}$ is an algebraic
operator whose minimal polynomial takes the form
$(x-z)^n$ for some $n\Le k$ (Recall that an operator
$S\in \lino{\mscr}$ is said to be {\em algebraic} if
there exists a nonzero polynomial $p \in \cbb[x]$ such
that $p(S)=0$.)

Given $T\in \lino{\mscr}$, $h\in \mscr\setminus\{0\}$
and $z\in \cbb$, we say that $T$ admits a {\em Jordan
block} $J$ (or that $J$ is a {\em Jordan block} of
$T$\/) at $h$ with eigenvalue $z$ if
$J=T|_{\cscr_T(h)}$ and $J-z I_{\cscr_T(h)}$ is a
nilpotent operator. Notice that $T$ admits a Jordan
block at $h$ with eigenvalue $z$ if and only if $h\in
\gev{T}{z} \setminus \{0\}$. Moreover, if this is the
case, then by the spectral mapping theorem and
\eqref{cyc2}, $\sigma(T|_{\cscr_T(h)}) = \{z\}$ and
$z$ is an eigenvalue of $T$ (recall that $\dim
\cscr_T(h) < \infty$). This justifies that part of the
definition of a Jordan block which refers to the
expression ``with eigenvalue $z$''.

Suppose that $T$ admits a Jordan block $J$ at $h$ with
eigenvalue $z$. Set $N=J-zI_{\cscr_T(h)}$ and
$k=\nil{N}$. Then $\{0\} \varsubsetneq \jd{N}
\varsubsetneq \ldots \varsubsetneq \jd{N^k} = \mscr$.
Since $\dim \cscr_T(h) = k$, it must be that $\dim
\jd{N^j}=j$ for $j=1,\ldots,k$. Hence, there exists a
Hamel basis $\{e_j\}_{j=1}^k$ of $\cscr_T(h)$ such
that $N^j e_j=0$ for $j=1,\ldots,k$. Then the matrix
representation of $J$ with respect to this basis takes
the familiar form (with zero entries in blank spaces)
   \begin{align*}
   \left[
   \begin{matrix}
z & 1 & & &
   \\
 & z & 1 & &
   \\
& & \ddots & \ddots&
   \\
& & & z & 1
   \\
& & & & z
 \end{matrix} \right].
   \end{align*}
Moreover, by Proposition~\ref{Bang1}, $J$ is an
$m$-isometry for some $m\in \nbb$ if and only if
$|z|=1$, and if this is the case, then $J$ is a strict
$(2 \nil{N}-1)$-isometry.

Theorem~\ref{alg-m-iso-lem-b} below, which essentially
follows from corollary to Theorem~\ref{alg-m-iso-lem},
characterizes ``orthogonality'' of Jordan blocks
corresponding to distinct eigenvalues of modulus $1$.
The most spectacular is the implication
(ii)$\Rightarrow$(i).
   \begin{thm} \label{alg-m-iso-lem-b}
Suppose $T\in\lino{\mscr}$ admits a Jordan block at
$h_j\in \mscr$ with eigenvalue $z_j\in \tbb$ for
$j=1,2$ and $z_1\neq z_2$. Then the following
conditions are~equivalent{\em :}
   \begin{enumerate}
   \item[(i)] $\cscr_T(h_1) \perp \cscr_T(h_2)$,
   \item[(ii)] either $z_1=-z_2$ and there is
$(\epsilon_1,\epsilon_2) \in\ubb$ such that
$\|T^n(\epsilon_k T^j h_1 + h_2)\|^2$ is a polynomial
in $n$ for $k=1,2$ and every $j\in \zbb_+$, or
$z_1\neq -z_2$ and $\|T^n(T^jh_1 + h_2)\|^2$ is a
polynomial in $n$ for every $j\in \zbb_+$,
   \item[(iii)] $\|T^n(g_1 + h_2)\|^2$ is a
polynomial in $n$ for every $g_1\in \cscr_T(h_1)$,
   \item[(iv)] $\|T^n(g_1 + g_2)\|^2$ is a
polynomial in $n$ for all $g_1\in \cscr_T(h_1)$ and
$g_2\in \cscr_T(h_2)$,
   \item[(v)] there is $m\in \nbb$ such that $T|_{\cscr_T(h_1) +
\cscr_T(h_2)}$ is an $m$-isometry.
   \end{enumerate}
   \end{thm}
   \begin{proof}
The implication (v)$\Rightarrow$(iv) follows from
Proposition~\ref{polar-1}. The implications
(iv)$\Rightarrow$(iii) and (iii)$\Rightarrow$(ii) are
obvious.

(ii)$\Rightarrow$(i) Apply Corollary~\ref{Bang-ch}.

(i)$\Rightarrow$(v) By assumption and
Proposition~\ref{Bang1}, $T|_{\cscr_T(h_j)}$ is
$m_j$-isometry, where $m_j=2 \nil{N_j}-1$ with
$N_j:=T|_{\cscr_T(h_j)}-z_j I_{\cscr_T(h_j)}$ for $j=1,2$.
Hence $T|_{\cscr_T(h_j)}$is an $m$-isometry with
$m=\max\{m_1,m_2\}$ for $j=1,2$. Using \eqref{def-miso},
(i) and the fact that the vector spaces $\cscr_T(h_1)$,
$\cscr_T(h_2)$ and $\cscr_T(h_1) + \cscr_T(h_2)$ are
invariant for $T$, we easily verify that the operator
$T|_{\cscr_T(h_1) + \cscr_T(h_2)}$ is an $m$-isometry. This
completes the proof.
   \end{proof}
   \begin{rem}
It is worth mentioning that some implications of
Theorem~\ref{alg-m-iso-lem-b} can also be proved in a
different way. Namely, the implication (iv)$\Rightarrow$(v)
is a direct consequence of Theorem~\ref{weakmiso}. In turn,
the implication (v)$\Rightarrow$(i) can be proved as
follows. Under the assumptions of
Theorem~\ref{alg-m-iso-lem-b}, suppose that
$T_0:=T|_{\mscr_0}$ is an $m$-isometry, where
$\mscr_0:=\cscr_T(h_1) + \cscr_T(h_2)$. Recall that
$\mscr_0$ is a finite dimensional vector space which is
invariant for $T$. Moreover, there exist $k_1,k_2 \in \nbb$
such that
   \begin{align} \label{rodo1}
\cscr_T(h_j)\subseteq \mscr_0 \cap \ker{(T - z_j
I)^{k_j}} = \ker{(T_0 - z_j I_{\mscr_0})^{k_j}}, \quad
j=1,2.
   \end{align}
Noticing that
   \begin{align*}
(T_0 - z_1 I_{\mscr_0})^{k_1}(T_0 - z_2
I_{\mscr_0})^{k_2}(g_1+g_2) \overset{\eqref{rodo1}}=
0, \quad g_1 \in \cscr_T(h_1), \, g_2 \in
\cscr_T(h_2),
   \end{align*}
we see that $T_0$ is an algebraic operator (observe that
$\cscr_T(h_1) \cap \cscr_T(h_2)=\{0\}$ due to \eqref{rodo1}
and \cite[Theorem~3.5.2]{Bi-So87}). Since $\mscr_0$ is
finite dimensional, $\mscr_0$ is a Hilbert space and
$T_0\in \ogr{\mscr_0}$. Define the polynomial $p\in
\cbb[x,y]$ by $p(x,y)=(xy-1)^{m}$. Then one can verify that
   \begin{align*}
p(T_0)=(-1)^m\bscr_m(T_0)=0,
   \end{align*}
where the operator $p(T_0)$ is defined as in
\cite[(1)]{A-H-S}. Hence $T_0$ is a root of $p$ in the
sense of \cite[Page 126]{A-H-S}. Note now that if
$\lambda,\mu\in \sigma(T_0)$ and $\lambda\neq \mu$,
then $p(\lambda,\bar \mu)\neq 0$. Indeed, otherwise
$p(\lambda,\bar \mu)= 0$, which implies that
$\lambda\bar \mu=1$. Since by \eqref{as-spec} and
$\dim \mscr_0 < \infty$, $\sigma(T_0) \subseteq \tbb$,
we deduce that $\lambda = \mu$, a contradiction.
Applying \cite[Lemma~19]{A-H-S}, we get
$\gev{T_0}{z_1} \perp \gev{T_0}{z_2}$. Since by
\eqref{rodo1}, $\cscr_T(h_j) \subseteq \gev{T_0}{z_j}$
for $j=1,2$, we conclude that $\cscr_T(h_1) \perp
\cscr_T(h_2)$, which gives (i).
   \hfill{$\diamondsuit$}
   \end{rem}
   Using Theorem~\ref{alg-m-iso-lem-b}, we can
completely characterize algebraic $m$-isometries on
inner product spaces. Note that the equivalence
(i)$\Leftrightarrow$(iv) of
Proposition~\ref{m-iso-alg} below was proved by
Berm\'{u}dez, Martin\'{o}n and Noda in the finite
dimensional case (see
\cite[Theorem~2.7]{Ber-Mar-No13}).
   \begin{pro} \label{m-iso-alg}
Suppose $\mscr\neq \{0\}$ and $T\in \lino{\mscr}$.
Then the following conditions are equivalent{\em :}
   \begin{enumerate}
   \item[(i)] $T$ is an $m$-isometric algebraic operator
for some $m\in\nbb$,
   \item[(ii)] $T$ is algebraic and $\|T^n h\|^2$ is a polynomial in $n$ for
every $h\in \mscr$,
   \item[(iii)] $T$ has a nontrivial orthogonal decomposition\/\footnote{\;The
orthogonal decomposition in \eqref{sta-orth-1}
corresponds to the orthogonal decomposition
$\mscr=\mscr_1\oplus \ldots \oplus \mscr_{\varkappa}$,
while nontriviality means that $\mscr_j\neq \{0\}$ for
all $j=1, \ldots, \varkappa$.}
   \begin{align} \label{sta-orth-1}
T=(z_1 I_{\mscr_1} + N_1) \oplus \ldots \oplus
(z_\varkappa I_{\mscr_\varkappa} + N_\varkappa) \qquad
(\varkappa\in \nbb),
   \end{align}
where $z_1, \ldots, z_\varkappa$ are distinct elements
of $\tbb$ and $N_1 \in \lino{\mscr_1}, \ldots,
N_\varkappa \in \lino{\mscr_\varkappa}$ are nilpotent
operators,
   \item[(iv)] $T$ is algebraic and there exist an
isometric $($equivalently, a unitary$)$ operator $V\in
\lino{\mscr}$ and a nilpotent operator $N\in
\lino{\mscr}$ such that $VN=NV$ and $T=V+N$.
   \end{enumerate}
Moreover, if {\em (iii)} holds, then $T$ is a strict
$m$-isometry with
   \begin{align} \label{maxnj}
m = \max_{j\in \{1, \ldots, \varkappa\}} \big(2
\nil{N_j}-1\big).
   \end{align}
   \end{pro}
   \begin{proof}
(i)$\Rightarrow$(ii) This implication is a direct
consequence of Proposition~\ref{polar-1}.

(ii)$\Rightarrow$(iii) Since $T$ is algebraic, there
exist $\varkappa\in \nbb$, $k_1, \ldots, k_\varkappa
\in \nbb$, distinct numbers $z_1, \ldots, z_\varkappa
\in \cbb$ and nonzero vector subspaces $\mscr_1,
\ldots, \mscr_\varkappa$ of $\mscr$ such that $\mscr=
\mscr_1 \dotplus \ldots \dotplus \mscr_\varkappa$ (a
direct sum), $T(\mscr_j) \subseteq \mscr_j$ and $(T
-z_j I)^{k_j}|_{\mscr_j} = 0$ for $j=1, \ldots,
\varkappa$ (see e.g., \cite[Section~6]{C-J-S09}).
Then, by Proposition~\ref{Bang1}(ii), $z_j\in \tbb$
for $j=1, \ldots, \varkappa$. Applying the implication
(iv)$\Rightarrow$(i) of Theorem~\ref{alg-m-iso-lem-b}
shows that $\mscr_i \perp \mscr_j$ for all $i\neq j$,
which yields (iii).

(iii)$\Rightarrow$(i) Set $k_j=\nil{N_j}$ for
$j=1,\ldots,\varkappa$. Since $p(T)=0$ for $p\in \cbb[x]$
given by
   \begin{align*}
p=(x-z_1)^{k_1} \cdots (x-z_\varkappa)^{k_\varkappa},
   \end{align*}
the operator $T$ is algebraic. In turn, by
Proposition~\ref{Bang1}(i), $T_j:=z_j I_{\mscr_j} +
N_j$ is an $m$-isometry for $j=1, \ldots, \varkappa$,
where $m$ is as in \eqref{maxnj}. Hence by
\eqref{def-miso} and \eqref{sta-orth-1}, $T$ is an
$m$-isometry. We now show that $T$ is a strict
$m$-isometry. Let $j_0\in \{1, \ldots,\varkappa\}$ be
such that $m=2k_{j_0} -1$. Then by \eqref{def-miso}
and Proposition~\ref{Bang1}(i) applied to $T_{j_0}$,
there exists $h_{j_0}\in \mscr_{j_0}$ such that
$\hat\fcal_{T_{j_0};m-1}(h_{j_0}) \neq 0$ which
together with \eqref{sta-orth-1} implies that
$\hat\fcal_{T;m-1}(g_{0}) \neq 0$, where $g_0:=g_{0,1}
\oplus \ldots \oplus g_{0,\varkappa}$ with $g_{0,j}=0$
for all $j \neq j_0$ and $g_{0,j_0}=h_{j_0}$. This
shows that the operator $T$ is not an
$(m-1)$-isometry.

(iii)$\Rightarrow$(iv) Clearly $V:=z_1 I_{\mscr_1}
\oplus \ldots \oplus z_\varkappa I_{\mscr_\varkappa}$
and $N:= N_1 \oplus \ldots \oplus N_\varkappa$
satisfy~(iv).

(iv)$\Rightarrow$(i) It suffices to apply
Proposition~\ref{Bang1}(i).
   \end{proof}
   \begin{rem} \label{inv-miso}
In view of Proposition~\ref{m-iso-alg}, if $T\in
\ogr{\hh}$ is an algebraic strict $m$-isometry, then
$m$ is a positive odd number\footnote{\;The oddness of
$m$ can also be deduced from
\cite[Proposition~1.23]{Ag-St1} and the inclusion
$\sigma(T) \subseteq \tbb$ which follows from
Proposition~\ref{m-iso-alg}.}. In particular, this is
the case for strict $m$-isometries on finite
dimensional Hilbert spaces (use the Cayley-Hamilton
theorem). According to Proposition~\ref{m-iso-alg},
for any $\epsilon \in (0,\infty)$ and for any positive
odd number $m$, there exists a strict $m$-isometry on
a finite dimensional Hilbert space (of dimension at
least $\frac{1}{2}(m+1)$) such that $\|T\| \Le 1 +
\epsilon$.
   \hfill{$\diamondsuit$}
   \end{rem}
   \begin{rem} \label{wiel-nod}
The implication (ii)$\Rightarrow$(i) of Proposition~
\ref{m-iso-alg} is not true if the assumption on
algebraicity is dropped. It suffices to consider the
countable orthogonal sum $\bigoplus_{n=1}^{\infty}
T_n$ on ``finite vectors'', where each $T_n$ is a
strict $m_n$-isometry with $m_n \to \infty$ as $n\to
\infty$. If $\{m_n\}_{n=1}^{\infty}$ are positive odd
numbers, then in view of Remark~\ref{inv-miso} we can
also guarantee that $\sup_{n\in \nbb} \|T_n\| <
\infty$, or equivalently that $T$ is continuous.
   \hfill{$\diamondsuit$}
   \end{rem}
   As shown below the only compact $m$-isometries are
algebraic operators on finite dimensional Hilbert
spaces, and as such are described by
Proposition~\ref{m-iso-alg}.
   \begin{pro} \label{compop}
Let $T\in \ogr{\hh}$ be an $m$-isometry. Suppose $T^l$ is
compact for some $l\in \nbb$. Then $\dim \hh < \infty$ and
$T$ is an algebraic $m$-isometry.
   \end{pro}
   \begin{proof}
Clearly, we can assume that $\hh\neq \{0\}$. Since, by
\cite[Theorem~2.3]{Ja}, $T^l$ is an $m$-isometry, there is
no loss of generality in assuming that $T$ is a compact
$m$-isometry. In view of the Riesz-Schauder theorem
\cite[Theorem~VI.15]{R-S}, the only possible accumulation
point for the spectrum of a compact operator is $0$.
Therefore by \eqref{as-spec}, $\sigma(T)\subseteq \tbb$.
This implies that $T$ is invertible in $\ogr{\hh}$. Since
$T$ is compact, we deduce that $\dim \hh < \infty$. By the
Cayley-Hamilton theorem, $T$ is algebraic. This completes
the proof.
   \end{proof}
   \begin{rem}
By Proposition~\ref{compop}, if $T\in \ogr{\hh}$ is an
$m$-isometry such that $T^l$ is compact for some $l\in
\nbb$, then $T$ is compact. This is not true for
operators which are not $m$-isometries. Indeed, if
$\dim{\hh} \Ge \aleph_0$, then for any integer $l\Ge
2$ there exists an operator $T\in \ogr{\hh}$ such that
$T^l$ is compact, though the operators $T^1, \ldots,
T^{l-1}$ are not, e.g., the nilpotent operator $T\in
\ogr{\hh}$ defined by
   \begin{align*}
T(h_1 \oplus \ldots \oplus h_l) = 0 \oplus h_1 \oplus
\ldots \oplus h_{l-1}, \quad h_1, \ldots, h_l \in
\mathcal M,
   \end{align*}
on the orthogonal sum $\hh = \mathcal M \oplus \ldots
\oplus \mathcal M$ of $l$ copies of an infinite
dimensional Hilbert space $\mathcal M$.
   \hfill{$\diamondsuit$}
   \end{rem}
   \textbf{Acknowledgement}. A substantial part of this
paper was written while the first and the third author
visited Kyungpook National University during the autumn of
2018 and the spring of 2019. They wish to thank the faculty
and the administration of this unit for their warm
hospitality.

   \bibliographystyle{amsalpha}
   
   \end{document}